\documentclass[preprint, 10pt, english]{elsarticle}
\usepackage{amsthm}
\usepackage{amsmath}
\usepackage{latexsym, amssymb}
\usepackage{txfonts}
\usepackage{mathtools}
\usepackage{color}
\usepackage{babel}
\usepackage[all]{xy}
\usepackage[capitalise]{cleveref}

\newtheorem{thm}{Theorem}[section] 

\newtheorem{cor}[thm]{Corollary}
\newtheorem{defn}[thm]{Definition}

\newtheorem{lem}[thm]{Lemma}
\newtheorem{prop}[thm]{Proposition}

\newtheorem{ques}[thm]{Question}

\theoremstyle{definition}
\newtheorem{rem}[thm]{Remark}
\newtheorem{exmpl}[thm]{Example}

\newcommand\operA[2]{{\if!#2!\operatorname{#1}\else{\operatorname{#1}_{#2}^{\phantom{I}}}\fi}} 

\def\dim{{\operatorname{dim}}}

\newcommand{\Trace}[1][]{\if!#1!\operatorname{Tr}\else{\operatorname{Tr}_{#1}^{\phantom{I}}}\fi} 

\long\def\forget#1\forgotten{{}} %

\def\({\left(}
\def\){\right)}

\newcommand\LAY[3][]{{\begin{array}{c}\mbox{#2} \if#1!{}\else{+}\fi \\ \mbox{#3}\end{array}}}

\makeatletter

\def\ps@pprintTitle{%
 \let\@oddhead\@empty
 \let\@evenhead\@empty
 \def\@oddfoot{}%
 \let\@evenfoot\@oddfoot}

\newcommand{\bigperp}{%
  \mathop{\mathpalette\bigp@rp\relax}%
  \displaylimits
}

\newcommand{\bigp@rp}[2]{%
  \vcenter{
    \m@th\hbox{\scalebox{\ifx#1\displaystyle2.1\else1.5\fi}{$#1\perp$}}
  }%
}
\makeatother

\newcommand{\Om}{\Omega}

\newcommand{\wg}{\wedge}
\newcommand{\bwg}{\bigwedge}

\renewcommand{\geq}{\geqslant}
\renewcommand{\leq}{\leqslant}

\DeclareMathOperator{\coker}{coker}

\newif\iffurther
\furtherfalse

\journal{??}

\begin{document}
\begin{frontmatter}

\title{Kato-Milne Cohomology and Polynomial Forms}

\author{Adam Chapman}
\ead{adam1chapman@yahoo.com}
\address{Department of Computer Science, Tel-Hai Academic College, Upper Galilee, 12208 Israel}
\author{Kelly McKinnie}
\ead{kelly.mckinnie@mso.umt.edu}
\address{Department of Mathematics, University of Montana, Missoula, MT 59812, USA}

\begin{abstract}
Given a prime number $p$, a field $F$ with $\operatorname{char}(F)=p$ and a positive integer $n$, we study the class-preserving modifications of Kato-Milne classes of decomposable differential forms.
These modifications demonstrate a natural connection between differential forms and $p$-regular forms.
A $p$-regular form is defined to be a homogeneous polynomial form of degree $p$ for which there is no nonzero point where all the order $p-1$ partial derivatives vanish simultaneously. We define a $\widetilde C_{p,m}$ field to be a field over which every $p$-regular form of dimension greater than $p^m$ is isotropic. The main results are that for a $\widetilde C_{p,m}$ field $F$, the symbol length of $H_p^2(F)$ is bounded from above by $p^{m-1}-1$ and for any $n \geq \lceil (m-1) \log_2(p) \rceil+1$, $H_p^{n+1}(F)=0$.
\end{abstract}

\end{frontmatter}

\section{Introduction}

In this paper we study the connection between the Kato-Milne cohomology groups $H_p^{n+1}(F)$ over a field $F$ with $\operatorname{char}(F)=p$ for some prime integer $p$, and homogeneous polynomial forms of degree $p$ over $F$.
The three main objectives of this work are:
\begin{enumerate}
\item Finding a number $n_0$ such that for any $n \geq n_0$, $H_p^{n+1}(F)=0$.
\item Finding an upper bound for the symbol length of $H_p^2(F)$, which in turn provides an upper bound for the symbol length of $\prescript{}{p}Br(F)$.
\item Finding a number $s$ such that any collection of $s$ inseparably linked decomposable differential forms in $H_p^{n+1}(F)$ are also separably linked.
\end{enumerate}

\subsection{The Kato-Milne Cohomology Groups}

Given a prime number $p$ and a field $F$ of $\operatorname{char}(F)=p$, we consider the space of absolute differential forms 
$\Om_F^1$, which is defined to be the
$F$-vector space generated by the symbols $da$ subject to the relations $d(a+b)=da+db$ and $d(ab)=adb+bda$ for any $a,b \in F$.
The space of $n$-differential forms $\Om_F^n$ for any positive integer $n$ is then defined by the
$n$-fold exterior power 
$\Om_F^n=\bwg^n(\Om_F^1)$, which is consequently an $F$-vector space spanned
by  $da_1\wg\ldots\wg da_n$, $a_i\in F$. The derivation $d$ extends to an 
operator $d\,:\,\Om_F^n \to \Om_F^{n+1}$ by $d(a_0da_1\wg\ldots\wg da_n)=
da_0\wg da_1\wg\ldots\wg da_n$.  We define $\Om_F^0=F$, $\Om_F^n=0$ for $n<0$, and
$\Om_F=\bigoplus_{n\geq 0}\Om_F^n$,  the algebra of differential forms
over $F$ with multiplication naturally defined by
$$(a_0da_1\wg\ldots\wg da_n)(b_0db_1\wg\ldots\wg db_m)=
a_0b_0da_1\wg\ldots\wg da_n\wg db_1\wg\ldots\wg db_m\,.$$ 

There exists a well-defined group homomorphism $\Om_F^n\to \Om_F^n/d\Om_F^{n-1}$, the
Artin-Schreier map $\wp$, which acts on decomposable differential forms as follows:
$$\alpha\frac{d \beta_1}{\beta_1}\wg\ldots\wg \frac{d \beta_n}{\beta_n}\,\longmapsto\,
(\alpha^p-\alpha)\frac{d \beta_1}{\beta_1}\wg\ldots\wg \frac{d \beta_n}{\beta_n}.$$
The group $H_p^{n+1}(F)$ is defined to be $\coker(\wp)$.
By \cite{Kato:1982}, in the case of $p=2$, there exists an isomorphism 
\begin{eqnarray*}
H_2^{n+1}(F) &\stackrel{\cong}{\longrightarrow} & I_q^{n+1}(F)/I_q^{n+2}(F), \enspace \text{given by}\\
\alpha \frac{d \beta_1}{\beta_1}\wg\ldots\wg \frac{d \beta_n}{\beta_n}
& \longmapsto &   \langle \langle \beta_1,\dots,\beta_n,\alpha]]  \mod I_q^{n+2}(F)
\end{eqnarray*}
where $\langle \langle \beta_1,\dots,\beta_n,\alpha]]$ is a quadratic $n$-fold Pfister form.

By \cite[Section 9.2]{GilleSzamuely:2006}, when $n=1$, there exists an isomorphism 
\begin{eqnarray*}
H_p^2(F) &\stackrel{\sim}{\longrightarrow}& \prescript{}{p} Br(F), \enspace \text{given by}\\
\alpha \frac{d\beta}{\beta} &\longmapsto & [\alpha,\beta)_{p,F},
\end{eqnarray*}
where $[\alpha,\beta)_{p,F}$ is the degree $p$ cyclic $p$-algebra 
$$F \langle x,y : x^p-x=\alpha, y^p=\beta, y x y^{-1}=x+1 \rangle.$$

In the special case of $p=2$ and $n=1$, these cyclic $p$-algebras are quaternion algebras $[\alpha,\beta)_{2,F}$ that can be identified with their norm forms which are quadratic 2-fold Pfister forms $\langle \langle \beta,\alpha]]$ (see \cite[Corollary 12.2 (1)]{EKM}).

\subsection{$C_m$ and $\widetilde C_{p,m}$ Fields}\label{up_invariant}

A $C_m$ field is a field $F$ over which every homogeneous polynomial form of degree $d$ in more than $d^m$ variables is isotropic (i.e. has a nontrivial zero).
It was suggested in \cite[Chapter II, Section 4.5, Exercise 3 (b)]{Serre} that if $F$ is a $C_m$ field with $\operatorname{char}(F) \neq p$ then for any $n \geq m$, $H^{n+1}(F,\mu_p^{\otimes n})=0$.
This fact is known for $p=2$ because of the Milnor conjecture, proven in \cite{Voevodsky}.
It was proven in \cite{KrashenMatzri:2015} that for any prime $p>3$, $C_m$ field $F$ with $\operatorname{char}(F) \neq p$ and $n \geq \lceil (m-2) \log_2(p)+1 \rceil$, we have $H^{n+1}(F,\mu_p^{\otimes n})=0$. (The same result holds when $p=3$ for $n\geq \lceil (m-3)\log_2(3)+3\rceil$.)
The analogous statement for fields $F$ with $\operatorname{char}(F)=p$ is that if $F$ is a $C_m$ field then $H_p^{n+1}(F)=0$ for every $n \geq m$. This is true, as stated in \cite[Chapter II, Section 4.5, Exercise 3 (a)]{Serre} and proven explicitly in \cite{ArasonBaeza:2010}. It follows from the fact that $C_m$ fields $F$ have $p$-rank at most $m$, i.e. $[F:F^p] \leq p^m$.
We consider a somewhat different property of fields that avoids directly bounding their $p$-rank. We say that a homogeneous polynomial form of degree $p$ over $F$ is {\bf \boldmath{$p$}-regular} if there is no nonzero point where all the partial derivatives of order $p-1$ vanish. We denote by $u_p(F)$ the maximal dimension of an anisotropic $p$-regular form over $F$. We say $F$ is a $\widetilde{C}_{p,m}$ field if $u_p(F) \leq p^m$.
We prove that if $F$ is $\widetilde{C}_{p,m}$ then for any $n \geq \lceil (m-1) \log_2(p) \rceil+1$, we have $H_p^{n+1}(F)=0$.
(See Section \ref{tildeCm} for examples of $\widetilde C_{p,m}$ which are not $C_m$.)

\begin{rem}
Note that when $p=2$, the notion of a $p$-regular form coincides with nonsingular quadratic form, and $u_p(F)$ boils down to the $u$-invariant $u(F)$ of $F$. In this case, $\lceil (m-1) \log_2(p) \rceil+1=m$, which recovers the known fact that when $u(F) \leq 2^m$, we have $H_2^{m+1}(F) \cong I_q^{m+1}(F)/I_q^{m+2}(F)=0$.
\end{rem}

\subsection{Symbol Length in $H_p^2(F)$}

By \cite[Theorem 30]{Albert:1968} (when $\operatorname{char}(F)=p$) and \cite{MS} (when $\operatorname{char}(F) \neq p$ and $F$ contains a primitive $p$th root of unity), $\prescript{}{p} Br(F)$ is generated by cyclic algebras of degree $p$.
The symbol length of a class in $\prescript{}{p} Br(F)$ is the minimal number of cyclic algebras required in order to express this class as a tensor product of cyclic algebras. The symbol length of $\prescript{}{p} Br(F)$ is the supremum of the symbol length of all the classes in $\prescript{}{p} Br(F)$. Recall that when $\operatorname{char}(F)=p$, $\prescript{}{p} Br(F) \cong H_p^2(F)$.

It was shown in \cite[Corollary 3.3]{Chapman:2017} that if the maximal dimension of an anisotropic form of degree $p$ over $F$ is $d$ then the symbol length of $\prescript{}{p} Br(F)$ is bounded from above by $\left \lceil \frac{d-1}{p} \right \rceil-1$, providing a characteristic $p$ analogue to a similar result obtained in \cite{Matzri:2016} in the case of $\operatorname{char}(F) \neq p$. As a result, if $F$ is $C_m$ then $d \leq p^m$ and so this upper bound boils down to $p^{m-1}-1$.
However, the symbol length of $\prescript{}{p} Br(F)$ when $F$ is a $C_m$ field with $\operatorname{char}(F)=p$ is bounded from above by the $p$-rank which is at most $m$ (see \cite[Theorem 28]{Albert:1968}).
We show that the forms discussed in \cite{Chapman:2017} are actually $p$-regular forms, which gives the upper bound $\left \lceil \frac{u_p(F)-1}{p} \right \rceil-1$ for the symbol length (which coincides with $\frac{u(F)}{2}-1$ when $p=2$ as in \cite[Corollary 4.2]{Chapman:2017}). In particular, if $F$ is $\widetilde C_{p,m}$ then the symbol length is bounded from above by $p^{m-1}-1$. (This bound is in fact sharp for $p=2$ as proven in \cite[Proposition 4.5]{Chapman:2017}.)

\subsection{Separable and Inseparable Linkage}\label{SepInsLinkage}

A differential form in $H_p^{n+1}(F)$ is called ``decomposable" if it can be written as
$\alpha \frac{d \beta_1}{\beta_1} \wedge \dots \wedge \frac{d \beta_n}{\beta_n}$
for some $\alpha \in F$ and $\beta_1,\dots,\beta_n \in F^\times$.
We say that a collection of decomposable differential forms $\omega_1,\dots,\omega_m$ in $H_p^{n+1}(F)$ are inseparably $\ell$-linked if they can be written as 
$$\omega_i=\alpha_i \frac{d \beta_{i,1}}{\beta_{i,1}} \wedge \dots \wedge \frac{d \beta_{i,n}}{\beta_{i,n}}, \enspace i \in \{1,\dots,m\}$$
such that $\beta_{1,k}=\dots=\beta_{m,k}$ for all $k \in \{1,\dots,\ell\}$.
We say they are separably $\ell$-linked if they can be written in a similar way such that $\alpha_1=\dots=\alpha_m$ and $\beta_{1,k}=\dots=\beta_{m,k}$ for all $k \in \{1,\dots,\ell-1\}$.
By the identification of decomposable differential forms with quaternion algebras (when $p=2$ and $n=1$), cyclic $p$-algebras (when $n=1$) and quadratic Pfister forms (when $p=2$), the notions of inseparable and separable linkage coincide with the previously defined notions of separable and inseparable linkages for these objects.
In \cite{Draxl:1975} it was proven that inseparable (1-)linkage of pairs of quaternion algebras implies separable (1-)linkage as well.
A counterexample to the converse was given in \cite{Lam:2002}.
These results extend naturally to Hurwitz algebras (\cite{ElduqueVilla:2005}) and quadratic Pfister forms (\cite[Corollary 2.1.4]{Faivre:thesis}).
In \cite[Corollary 5.4]{ChapmanGilatVishne:2017} it was shown that when $H_2^{n+2}(F)=0$, separable $n$-linkage and inseparable $n$-linkage for pairs of quadratic $(n+1)$-fold Pfister forms are equivalent.
In \cite{Chapman:2015} it was proven that inseparable $(1-)$linkage for pairs of cyclic $p$-algebras of degree $p$ implies separable $(1-)$linkage as well, and that the converse is not necessarily true.
It follows immediately that if two decomposable differential forms in $H_p^{n+1}(F)$ are inseparably $k$-linked then they are also separably $k$-linked. In this paper we generalize this statement for larger collections of forms: every collection of $1+\sum_{i=\ell}^n 2^{i-1}$ inseparably $n$-linked decomposable differential forms in $H_p^{n+1}(F)$ are also separably $\ell$-linked. In particular, this means that if three octonion algebras share a biquadratic purely inseparable field extension of $F$, then they also share a quaternion subalgebra.

\section{Fields with bounded $u_p$-invariant}\label{tildeCm}

There are examples in the literature of fields with $u(F)<2^{n+1}$ and unbounded $2$-rank (see \cite{MammoneTignolWadsworth:1991}). In particular, these fields are $\widetilde{C}_{2,m}$ but not $C_m$.

\begin{ques}
Is there a similar construction of $\widetilde{C}_{p,m}$ fields which are not $C_m$ for prime numbers $p>2$?
\end{ques}

The following construction gives an example of a field $F$ which is $\widetilde C_{p,0}$ but clearly not $C_0$:

\begin{exmpl}
Let $K$ be a field of characteristic $p$, $L=K(\lambda_1,\dots,\lambda_n)$ be the function field in $n$ algebraically independent variables ($n$ can also be $\infty$), and $F=L^{\operatorname{sep}}$ the separable closure of $L$.
Then $F$ is $\widetilde{C}_{p,0}$ and not $C_0$.
\end{exmpl}

\begin{proof}
Clearly it is not $C_0$ because its $p$-rank is at least $n$.
To show that $F$ is $\widetilde{C}_{p,0}$ it is enough to show that every $p$-regular form $\varphi(x_1,\dots,x_m)$ of dimension $m \geq 1$ over $F$ is isotropic.
Let $\varphi(x_1,\dots,x_m)$ be a $p$-regular form of dimension $m$ over $F$.
Since there are no $p$-regular forms of dimension 1, $m>1$.
Since $\varphi$ is $p$-regular, there exists a term with mixed variables and nonzero coefficient.
Without loss of generality, assume the power of $x_1$ in this term is $d$ where $1 \leq d \leq p-1$.
Write $\varphi$ as a polynomial in $x_1$ and coefficients in $F[x_2,\dots,x_m]$: $\varphi=c_p x_1^p+\dots+c_1 x_1+c_0$.
The coefficient $c_d$ is a nonzero homogeneous polynomial form of degree $p-d$ in $m-1$ variables. Since it is nonzero, we have $c_d(a_2,\dots,a_m) \neq 0$ for some $a_2,\dots,a_m \in F$, not all zero. Without loss of generality, assume that $a_2 \neq 0$, which means we could assume $a_2=1$. Then $c_d(x_2,a_3 x_2,\dots,a_m x_2)$ is a nonzero 1-dimensional form. Hence $\varphi(x_1,x_2,a_3 x_2,\dots,a_m x_2)$ is a nondiagonal 2-dimensional form of degree $p$. We shall explain now why this form must be isotropic:
Suppose it is anisotropic. Consider the polynomial $\varphi(x_1,1,a_3,\dots,a_m)$. This is a polynomial of degree $\leq p$ and at least $d$. If its degree is smaller than $p$ then since $F$ is separably closed, the polynomial decomposes into linear factors over $F$, and then it has a root in $F$, which means that $\varphi$ is isotropic. Assume the degree is $p$. Since it is of degree $p$, by \cite[Chapter V, Corollary 6.2]{Lang:2002} this field extension must be either separable or purely inseparable.
It cannot be purely inseparable because it has a nonzero term besides the degree $p$ and $0$ terms. Therefore it must be separable, but that contradicts the fact that $F$ is separably closed.
\end{proof}

One can construct fields $F$ with bounded $u_p(F)$ and infinite $p$-rank in the following way:

\begin{lem}
Let $K$ be a field of characteristic $p$ with $p$-rank $r$.
Then for any anisotropic $p$-regular form $\varphi(x_1,\dots,x_n)$ of dimension $n$, the function field $K(\varphi)=K(x_1,\dots,x_n : \varphi(x_1,\dots,x_n)=0)$ has $p$-rank $r+n-1$.
One can also take $r=\infty$ and then the $p$-rank of $K(\varphi)$ is $\infty$ as well.
\end{lem}

\begin{proof}
The function field $K(\varphi)$ of $\varphi$ is a degree $p$ separable extension of the function field $K(x_1,\dots,x_{n-1})$ in $n-1$ algebraically independent variables over $K$.
The $p$-rank of $K(x_1,\dots,x_{n-1})$ is $r+n-1$ by \cite[Lemma 2.7.2]{FriedJarden}. By \cite[Lemma 2.7.3]{FriedJarden} the $p$-rank of $K(\varphi)$ is also $r+n-1$.
\end{proof}

\begin{cor}
Let $K$ be a field of characteristic $p$ and $p$-rank $r$, and let $M$ be a positive integer.
Then $K$ is a subfield of some field $F$ with $p$-rank at least $r$ and $u_p(F) \leq M$.
\end{cor}

\begin{proof}
This $F$ is taken to be the compositum of the function fields of all anisotropic $p$-regular forms of dimension greater than $M$. By the previous lemma, the $p$-rank of $F$ is at least $r$. Clearly $u_p(F) \leq M$.
\end{proof}

The last corollary provides examples of fields $F$ which are $\widetilde C_{p,m}$ by taking $M=p^m$.
The fact that if $F$ is $C_m$ then the field of Laurent series $F((\lambda))$ over $F$ is $C_{m+1}$ does not hold for $\widetilde{C}_{p,m}$ fields in general.
For example, if one takes one of the fields constructed in \cite{MammoneTignolWadsworth:1991} with $u(F)=2^m<\hat{u}(F)$, then $F$ is $\widetilde{C}_{2,m}$. However, by \cite[Corollary 2.10]{Baeza:1982}, $u(F((\lambda)))=2\hat{u}(F)>2^{m+1}$, hence $F((\lambda))$ is not $\widetilde{C}_{2,m+1}$.

\section{Symbol Length and $p$-Regular forms}

In this section we describe certain properties of $p$-regular forms and make a note on the symbol length of classes in $\prescript{}{p}Br(F)$ when $F$ is a field of characteristic $p$ with bounded $u_p(F)$.

Let $p$ be a prime integer and $F$ be a field of characteristic $p$.
Let $V=F v_1+\dots+F v_m$ be an $m$-dimensional $F$-vector space.
A map $\varphi : V \rightarrow F$ is called a homogeneous polynomial form of degree $p$ if it satisfies
$$\varphi(a_1 v_1+\dots+a_m v_m)=\sum_{i_1+\dots+i_m=p} c_{i_1,\dots,i_m} a_1^{i_1} \dots a_m^{i_m}$$
for any $a_1,\dots,a_m \in F$ where $c_{i_1,\dots,i_m}$ are constants in $F$.
We say that $\varphi$ is isotropic if there exists a nonzero $v$ in $V$ such that $\varphi(v)=0$. Otherwise $\varphi$ is anisotropic.
We say $\varphi$ is $p$-regular if there is no $v \in V \setminus \{0\}$ for which all the order $p-1$ partial derivatives of $\varphi$ vanish.
The nonexistence of such points does not depend on the choice of basis.
In the special case of $p=2$, this notion coincides with nonsingularity.
In particular, diagonal forms of degree $p$ over $F$ are not $p$-regular.
Given a homogeneous polynomial form $\varphi : V \rightarrow F$, we can consider the scalar extension $\varphi_L$ of $\varphi$ from $F$ to $L$. 

\begin{lem}\label{closure}
Given a field extension $L/F$, if $\varphi_L$ is a $p$-regular form for some homogeneous polynomial form $\varphi$ of degree $p$ over $F$ then $\varphi$ is $p$-regular as well.
\end{lem}

\begin{proof}
Assume the contrary, that $\varphi$ is not $p$-regular, i.e. there exists $v \neq 0$ such that all the order $p-1$ partial derivatives of $\varphi$ vanish. Since the partial derivatives do not change under scalar extension, all the partial derivatives of $\varphi_L$ vanish at $v$, which means that $\varphi_L$ is not $p$-regular.
\end{proof}

\begin{defn}
Given homogeneous polynomial forms $\varphi : V \rightarrow F$ and $\phi : W \rightarrow F$ of degree $p$, we define the direct sum $\varphi \perp \phi$ to be the homogeneous polynomial form $\psi : V \oplus W \rightarrow F$ defined by $\psi(v+w)=\varphi(v)+\phi(w)$ for any $v \in V$ and $w \in W$.
\end{defn}

\begin{lem}\label{directsum}
The form $\varphi \perp \phi$ is $p$-regular if and only if both $\varphi$ and $\phi$ are $p$-regular.
\end{lem}

\begin{proof}
If $\varphi$ is not $p$-regular, there exists a nonzero $v \in V$ such that all the order $p-1$ partial derivatives of $\varphi$ vanish.
Then all the order $p-1$ partial derivatives of $\varphi \perp \phi$ vanish at the point $v \oplus 0$.

In the opposite direction, if $\varphi \perp \phi$ is not $p$-regular, then there exists a nonzero $v \oplus w$ where all the order $p-1$ partial derivatives of $\varphi \perp \phi$ vanish.
Without loss of generality, assume $v \neq 0$.
Then all the order $p-1$ partial derivatives of $\varphi \perp \phi$ vanish at $v$.
Since these derivatives are equal to the derivatives of $\varphi$ at $v$, $\varphi$ is not $p$-regular.
\end{proof}

\begin{lem}\label{fieldextension}
Given a separable field extension $L/F$ of degree $p$, the norm form $N : L \rightarrow F$ is a homogeneous polynomial form of degree $p$.
This form is $p$-regular.
\end{lem}

\begin{proof}
By Lemma \ref{closure} it is enough to show that the scalar extension of $N$ to the algebraic closure $\overline{F}$ of $F$ is $p$-regular.
Now, $L \otimes_F \overline{F}$ is $\underbrace{\overline{F} \times \dots \times \overline{F}}_{p \enspace \text{times}}$ and can be identified with the $p \times p$ diagonal matrices with entries in $\overline{F}$.
Therefore we have $$N_{\overline{F}}(a_1 e_1+a_2 e_2+\dots+a_p e_p)=a_1 a_2 \dots a_p.$$
The latter is clearly $p$-regular.
\end{proof}

\begin{rem}\label{scalar}
If $\varphi : V \rightarrow F$ is $p$-regular then for any nonzero scalar $c \in F$, $c \varphi$ defined by $(c\varphi)(v)=c \varphi(v)$ for any $v \in V$ is also $p$-regular. 
\end{rem}

\begin{lem}\label{twodim}
Given a prime number $p$ and a field $F$ with $\operatorname{char}(F) \geq p$ or $0$, the homogeneous polynomial form 
$$\varphi(a_1 v_1+a_2 v_2)=\alpha a_1^p-a_1 a_2^{p-1}+a_2^p$$
over the two-dimensional space $V=F v_1+F v_2$
is $p$-regular for any $\alpha \in F$.
\end{lem}

\begin{proof}
It is enough to note that the partial derivative obtained by differentiating $p-1$ times with respect to $a_2$ is $(p-1)! a_1$ and the partial derivative obtained by differentiating $p-2$ times with respect to $a_2$ and once with respect to $a_1$ is $-(p-1)! a_2$.
Therefore the only point where all the order $p-1$ partial derivatives vanish is $(0,0)$ and the form is $p$-regular.
\end{proof}

In \cite[Corollary 3.3]{Chapman:2017} it was proven that the symbol length of a $p$-algebra of exponent $p$ over $K$ is bounded by $\left \lceil \frac{d-1}{p} \right \rceil-1$ where $d$ is the maximal dimension of an anisotropic homogeneous polynomial form of degree $p$ over $K$.
Apparently $d$ can be replaced with $u_p(F)$.

\begin{thm}[{cf. \cite[Theorem 3.2 and Theorem 4.1]{Chapman:2017}}]\label{SymbolLength}
Let $p$ be a prime integer and let $F$ be a field with $\operatorname{char}(F) = p$ and finite $u_p(F)$.
Then every two tensor products $A=\bigotimes_{i=1}^m [\alpha_i,\beta_i)_{p,F}$ and $B=\bigotimes_{i=1}^\ell [\gamma_i,\delta_i)_{p,F}$ with $(m+\ell) p \geq u_p(F)-1$ can be changed such that $\alpha_1=\gamma_1$.
\end{thm}

\begin{proof}
It is enough to show that the homogeneous polynomial form considered in the proof of \cite[Theorem 3.2]{Chapman:2017} is $p$-regular.
This form $\varphi$ is the direct sum $\varphi' \perp \phi_1 \perp \dots \perp \phi_m \perp \psi_1 \perp \dots \perp \psi_\ell$ where 
$\varphi' : F \times F \rightarrow F$ is defined by $\varphi'(a,b)=(\sum_{i=1}^m(\alpha_i)-\sum_{i=1}^\ell(\gamma_i)) a^p-a^{p-1} b+b^p$,
for each $i \in \{1,\dots,m\}$, $\phi_i$ is $\beta_i N_{F[x : x^p-x=\alpha_i]/F}$, and for each $i \in \{1,\dots,\ell\}$, $\psi_i$ is $\gamma_i N_{F[x : x^p-x=\delta_i]/f}$.
By Lemma \ref{fieldextension} and Remark \ref{scalar}, the forms $\phi_1,\dots,\phi_m,\psi_1,\dots,\psi_\ell$ are $p$-regular.
By Lemma \ref{twodim}, $\varphi'$ is also $p$-regular.
Consequently, the form $\varphi$ is $p$-regular as direct sum of $p$-regular forms by Lemma \ref{directsum}.
\end{proof}

\begin{cor}[{cf. \cite[Corollary 3.3 and Corollary 4.2]{Chapman:2017}}]
Let $p$ be a prime integer and let $F$ be a field with $\operatorname{char}(F) = p$ and finite $u_p(F)$.
Then the symbol length in $H_p^2(F)$ is bounded from above by $\left \lceil \frac{u_p(F)-1}{p} \right \rceil -1$.
\end{cor}

\begin{proof}
It follows from Theorem \ref{SymbolLength} in the same manner \cite[Corollary 3.3]{Chapman:2017} follows from \cite[Theorem 3.2]{Chapman:2017}.
\end{proof}

\begin{cor}\label{Corup}
Let $p$ be a prime integer and let $F$ be a $\widetilde C_{p,m}$ field with $\operatorname{char}(F) = p$. Then the symbol length in $H_p^2(F)$ is bounded from above by $p^{m-1}-1$.
\end{cor}

\begin{proof}
An immediate result of the previous corollary, given that for a $\widetilde C_{p,m}$ field $F$, $u_p(F) \leq p^m$.
\end{proof}

\section{Symbol Length and $p$-Rank}

By \citep[Chapter 7, Theorem 28]{Albert:1968}, the $p$-rank of $F$ is an upper bound for the symbol length of $H_p^2(F)$. 
In fact, it can be shown that if the $p$-rank of $F$ is a finite integer $m$ then the symbol length of $H_p^{n+1}(F)$ is bounded from above by $\binom{m}{n}$ for any positive integer $n$.

\begin{rem}\label{remCm}
If $F$ is $C_m$ then its $p$-rank is $\leq m$. 
\end{rem}

\begin{proof}
Suppose $F$ is $C_m$ and that its $p$-rank is $n$.
Then $F=F^p v_1 \oplus \dots \oplus F^p v_{p^n}$ where $v_1,\dots,v_{p^n}$ are linearly independent over $F^p$.
Therefore the form $\varphi(a_1,\dots,a_{p^n})=v_1 a_1^p+\dots+v_{p^n} a_{p^n}^p$ is anisotropic. Since the dimension of $\varphi$ is $p^n$, $m$ must be $\geq n$.
\end{proof}

For a $C_m$ field (compared to $\widetilde C_{p,m}$) $F$, its $p$-rank
($\leq m$) provides a better upper bound for the symbol length of $H_p^2(F)$ than the upper bound given in Corollary \ref{Corup} ($p^{m-1}-1$).
The following proposition shows that there exist cases where the symbol length is actually equal to the $p$-rank:

\begin{prop}
Let $K=F((\beta_1))\dots((\beta_n))$ be the field of iterated Laurent series in $n$ variables over a perfect field $F$ of characteristic $p$.
Let $L/F$ be a $(\mathbb{Z}/p \mathbb{Z})^n$-Galois field extension given by $L=F[\wp^{-1}(\alpha_1),\dots,\wp^{-1}(\alpha_n)]$. Then the symbol length of the $p$-algebra 
$$D=[\alpha_1,\beta_1)_{p,K} \otimes \dots \otimes [\alpha_n,\beta_n)_{p,K}$$
is equal to the $p$-rank of $K$.\label{p1}
\end{prop}

\begin{proof}
The $p$-rank of $K$ in the proposition above is $n$ by \cite[Chapter 2, Lemma 2.7.2]{FriedJarden}.
The $p$-algebra $D$ is a generic abelian crossed product with maximal $(\mathbb{Z}/p \mathbb{Z})^n$-Galois subfield $K[\wp^{-1}(\alpha_1),\dots,\wp^{-1}(\alpha_n)]$, hence it is a division algebra (see \cite{AmitsurSaltman:1978}) and its symbol length is exactly $n$.
\end{proof}

\begin{rem}
In order to construct a Galois extension satisfying the conditions of Proposition \ref{p1}, take $\alpha_1,\ldots,\alpha_n$ to be algebraically independent variables over $\mathbb F_p$. Let $F$ be the perfect closure of $\mathbb F_p(\alpha_1,\ldots,\alpha_n)$. Then each field extension $F[\wp^{-1}(\alpha_i)]$ is $(\mathbb{Z}/p \mathbb{Z})$-Galois and as a set they are mutually linearly independent.  Therefore $L=F[\wp^{-1}(\alpha_1),\ldots,\wp^{-1}(\alpha_n)]$ is a $(\mathbb{Z}/p \mathbb{Z})^n$-Galois extension of the perfect field $F$. \label{r1}
\end{rem}

The following example presents a $C_m$ field with $p$-rank $n$ such that $m=2n+1$:

\begin{exmpl}
Let $F$ be the perfect closure of the function field $\mathbb{F}_p(\alpha_1,\dots,\alpha_n)$ as in remark \ref{r1}.
Let $K=F((\beta_1))\dots((\beta_n))$ be the field of iterated Laurent series in $n$ variables over $F$.
As mentioned above, the $p$-rank of $K$ is $n$, and the symbol length of the algebra $[\alpha_1,\beta_1)_{p,F} \otimes \dots \otimes [\alpha_n,\beta_n)_{p,F}$ is $n$.
The field $\mathbb{F}_p(\alpha_1,\dots,\alpha_n)$ is a $C_{n+1}$ field, and hence so is its perfect closure $F$.
Consequently, $K$ is a $C_{2n+1}$ field.
\end{exmpl}

\section{Class-Preserving Modifications of Decomposable Differential Forms}\label{Diff}

In this section we study the class-preserving modifications of decomposable differential forms in $H_p^{n+1}(F)$.
These modifications will be used in proving the main results of the following sections.
In the special case of $p=2$ they coincide with the known modifications of quadratic Pfister forms (see \cite{AravireBaeza:1992}).

\begin{lem}\label{calcs} Let $\omega=\alpha \frac{d \beta_1}{\beta_1}\wg\ldots\wg \frac{d\beta_n}{\beta_n}$ be a form in $H_p^{n+1}(F)$. Then:
\begin{enumerate}
\item[$(a)$] For any $i \in \{1,\dots,n\}$, $\omega=(\alpha+\beta_i) \frac{d \beta_1}{\beta_1}\wg\ldots\wg \frac{d\beta_n}{\beta_n}.$
\item[$(b)$]  For any $i \in \{1,\dots,n\}$ and nonzero $f \in F[\wp^{-1}(\alpha)]$, if $\operatorname{N}_{F[\wp^{-1}(\alpha)]/F}(f)=0$ then $\omega=0$. Otherwise, $\omega=\alpha \frac{d \beta_1}{\beta_1}\wg\ldots\wedge \frac{d(\operatorname{N}_{F[\wp^{-1}(\alpha)]/F}(f) \beta_i)}{\operatorname{N}_{F[\wp^{-1}(\alpha)]/F}(f) \beta_i} \wedge \dots\wg \frac{d\beta_n}{\beta_n}.$
\item[$(c)$] For any $i \in \{1,\dots,n\}$ and $\gamma \in F$, if $\beta_i+\gamma^p=0$ then $\omega=0$. Otherwise, there exists some $\alpha' \in F$ such that $\omega=\alpha' \frac{d \beta_1}{\beta_1}\wg\ldots\wedge \frac{d(\beta_i+\gamma^p)}{\beta_i+\gamma^p} \wedge \dots\wg \frac{d\beta_n}{\beta_n}.$
\item[$(d)$] For any distinct $i,j \in \{1,\dots,n\}$,
$$\omega=\alpha \frac{d \beta_1}{\beta_1} \wedge \dots \wedge \frac{d \beta_i}{\beta_i} \wedge \dots \wedge \frac{d (\beta_i \beta_j)}{\beta_i \beta_j} \wedge \dots \wedge \frac{d \beta_n}{\beta_n}.$$
\item[$(e)$] For any distinct $i,j \in \{1,\dots,n\}$, if $\beta_i+\beta_j=0$ then $\omega=0$. Otherwise,
$$\omega=\alpha \frac{d \beta_1}{\beta_1} \wedge \dots \wedge \frac{d (\beta_i+\beta_j)}{\beta_i+\beta_j} \wedge \dots \wedge \frac{d (\beta_i^{-1} \beta_j)}{\beta_i^{-1} \beta_j} \wedge \dots \wedge \frac{d \beta_n}{\beta_n}.$$
\item[$(f)$] For any distinct $i,j \in \{1,\dots,n\}$ and $f \in F[\wp^{-1}(\alpha)]$, if $\beta_i+\operatorname{N}_{F[\wp^{-1}(\alpha)]/F}(f)=0$ then $\omega=0$. Otherwise,
$$\omega=\alpha \frac{d \beta_1}{\beta_1} \wedge \dots \wedge \frac{d \beta_i}{\beta_i} \wedge \dots \wedge \frac{d ((\beta_i+\operatorname{N}_{F[\wp^{-1}(\alpha)]/F}(f)) \beta_j)}{(\beta_i+\operatorname{N}_{F[\wp^{-1}(\alpha)]/F}(f)) \beta_j} \wedge \dots \wedge \frac{d \beta_n}{\beta_n}.$$
\end{enumerate}
\end{lem}

\begin{proof}
Parts $(a), (b)$ and $(c)$ are elementary and follow immediately from the equivalent statements for cyclic $p$-algebras (see \cite[Lemma 2.2]{Chapman:2017}). Part $(d)$ follows from the linearity of logarithmic differential forms: $\frac{d (\beta_i \beta_j)}{\beta_i \beta_j}=\frac{d \beta_i}{\beta_i}+\frac{d \beta_j}{\beta_j}$.

For part $(e)$, if $\beta_i=-\beta_j$ then $d\beta_i \wedge d \beta_j=0$ and so $\omega=0$. Otherwise, it is enough to show that $$\frac{d \beta_i}{\beta_i} \wedge \frac{d \beta_j}{\beta_j}=\frac{d(\beta_i+\beta_j)}{\beta_i+\beta_j} \wedge \frac{d(\beta_i^{-1} \beta_j)}{\beta_i^{-1} \beta_j}.$$
To see this, notice that 
$d(\beta_i+\beta_j)\wedge d(\beta_i^{-1} \beta_j)=(\beta_i^{-1}+\beta_i^{-2} \beta_j) d \beta_i \wedge d \beta_j$ and divide both sides by $(\beta_i+\beta_j)(\beta_i^{-1} \beta_j)$.

For $(f)$, recall that for any $a \in F$ and $b \in F^\times$, $[a,b)_{p,F}$ is split if and only if $b$ is a norm in the \'{e}tale extension $F[\wp^{-1}(a)]/F$, and otherwise it is a division algebra. If $\beta_i+\operatorname{N}_{F[\wp^{-1}(\alpha)]/F}(f)=0$ then the algebra $[\alpha,\beta_i)_{p,F}=0$ by the norm condition, and so $\omega=0$. Otherwise, consider the form $$\alpha \frac{d \beta_1}{\beta_1} \wedge \dots \wedge \frac{d \beta_i}{\beta_i} \wedge \dots \wedge \frac{d ((\beta_i+\operatorname{N}_{F[\wp^{-1}(\alpha)]/F}(f)) \beta_j)}{(\beta_i+\operatorname{N}_{F[\wp^{-1}(\alpha)]/F}(f)) \beta_j} \wedge \dots \wedge \frac{d \beta_n}{\beta_n}.$$
By $(d)$, it is equal to the sum of $\omega$ and the form 
$$\alpha \frac{d \beta_1}{\beta_1} \wedge \dots \wedge \frac{d \beta_i}{\beta_i} \wedge \dots \wedge \frac{d (\beta_i+\operatorname{N}_{F[\wp^{-1}(\alpha)]/F}(f))}{\beta_i+\operatorname{N}_{F[\wp^{-1}(\alpha)]/F}(f)} \wedge \dots \wedge \frac{d \beta_n}{\beta_n}.$$
If $\operatorname{N}_{F[\wp^{-1}(\alpha)]/F}(f)=0$ then this form is clearly zero. Otherwise, by $(b)$ it is equal to
$$\alpha \frac{d \beta_1}{\beta_1} \wedge \dots \wedge \frac{d \operatorname{N}_{F[\wp^{-1}(\alpha)]/F}(f)^{-1}\beta_i}{\operatorname{N}_{F[\wp^{-1}(\alpha)]/F}(f)^{-1}\beta_i} \wedge \dots \wedge \frac{d (\operatorname{N}_{F[\wp^{-1}(\alpha)]/F}(f)^{-1}\beta_i+1)}{\operatorname{N}_{F[\wp^{-1}(\alpha)]/F}(f)^{-1}\beta_i+1} \wedge \dots \wedge \frac{d \beta_n}{\beta_n},$$
which is zero by $(c)$.
\end{proof}

Part (b) of Lemma \ref{calcs} shows that $\beta_n$ can be replaced by any nonzero element represented by the $p$-regular form $\varphi:F[\wp^{-1}(\alpha)] \to F$ with $f\mapsto \beta_n\operatorname{N}_{F[\wp^{-1}(\alpha)]/F}(f)$. In the following proposition we present a $p$-regular form of larger dimension whose values can also alter the last slot (at the possible cost of changing some of the other inseparable slots).

\begin{prop}\label{polynomialform}
Let $\alpha \in F$ and $\beta_1,\dots,\beta_n \in F^\times$, and write $\omega=\alpha \frac{d \beta_1}{\beta_1}\wg\ldots\wg \frac{d\beta_n}{\beta_n}$ for the corresponding form in $H_p^{n+1}(F)$.
For each $(d_1,\dots,d_n) \in \underbrace{\{0,1\} \times \dots \times \{0,1\}}_{n \enspace \text{times}}$, let $V_{d_1,\dots,d_n}$ be a copy of $F[\wp^{-1}(\alpha)]$ and $\varphi_{d_1,\dots,d_n} : V_{d_1,\dots,d_n} \rightarrow F$ be the homogeneous polynomial form of degree $p$ defined by $\varphi_{d_1,\dots,d_n}(f)=\operatorname{N}_{F[\wp^{-1}(\alpha)]/F}(f) \cdot \beta_1^{d_1} \cdot \ldots \cdot \beta_n^{d_n}$.
Write $$(\varphi,V)=\bigperp_{\begin{array}{r}0 \leq d_1,\dots,d_n \leq 1\\ (d_1,\dots,d_n) \neq (0,\dots,0)\end{array}} (\varphi_{d_1,\dots,d_n},V_{d_1,\dots,d_n}).$$
If there exists a nonzero $v$ such that $\varphi(v)=0$ then $\omega=0$. Otherwise, for every nonzero $v \in V$, there exist $\beta_1',\dots,\beta_{n-1}' \in F^\times$ such that :
$$\omega=\alpha \frac{d \beta_1'}{\beta_1'}\wg\ldots\wg \frac{d\beta_{n-1}'}{\beta_{n-1}'} \wedge \frac{d \varphi(v)}{\varphi(v)}.$$
\end{prop}

\begin{rem} To get a feel for the form $(\varphi,V)$, note that if we set $N=N_{F[\wp^{-1}(\alpha)]/F}$ then for $n=1$, $(\varphi,V) = (\varphi_1,V_1):F[\wp^{-1}(\alpha)] \to V$ with $x \mapsto N(x)\beta_1$. When $n=2$, $(\varphi,V):F[\wp^{-1}(\alpha)]^{\times 3} \to F$ with $(x,y,z) \mapsto N(x)\beta_1+N(y)\beta_2+N(z)\beta_1\beta_2$.\end{rem}

\begin{proof}[Proof of Proposition 5.2]
By induction on $n$.
The case of $n=1$ holds by the analogy to cyclic $p$-algebras.
Assume it holds for a certain $n-1$.
We shall show it holds also for $n$.
The vector space $V$ decomposes as $V_0 \oplus V_1$ where
$V_0=\bigperp_{0 \leq d_1,\dots,d_{n-1} \leq 1} V_{d_1,\dots,d_{n-1},0}$ and $V_1=\bigperp_{0 \leq d_1,\dots,d_{n-1} \leq 1} V_{d_1,\dots,d_{n-1},1}$. The latter decomposes as $V_1=V_{0,\dots,0,1}+V_1'$. There is a natural isomorphism $\tau : V_1' \rightarrow V_0$ identifying each $V_{d_1,\dots,d_{n-1},1}$ with $V_{d_1,\dots,d_{n-1},0}$. Under this homomorphism, we have $\varphi(v)=\beta_n \varphi(\tau(v))$ for any $v \in V_1'$.

Let $v$ be a nonzero vector in $V$.
Then $v=v_0+v_1'+v_{0,1}$ where $v_0 \in V_0$, $v_1' \in V_1'$ and $v_{0,1} \in V_{0,\dots,0,1} (\cong F[\wp^{-1}(\alpha)])$. 
Note that $\varphi(v)=\varphi(v_0)+\varphi(v_1')+\varphi(v_{0,1})=\varphi(v_0)+\beta_n (\varphi(\tau(v_1))+\operatorname{N}_{F[\wp^{-1}(\alpha)]/F}(v_{0,1}))$. Write $v_1=v_1'+v_{0,1}$. 
If $v_1'=0$ and $v_{0,1}=0$ then the statement follows immediately from the induction hypothesis. If $v_1'=0$ and $v_{0,1} \neq 0$ then the statement follows from the induction hypothesis and Lemma \ref{calcs} $(e)$.

Assume $v_1' \neq 0$. If $\varphi(\tau(v_1'))=0$ then by the induction hypothesis, $\omega=0$.
Otherwise, by the induction hypothesis, there exist $\beta_1',\dots,\beta_{n-2}'$ such that
$$\alpha \frac{d \beta_1}{\beta_1}\wg\ldots\wg \frac{d\beta_{n-2}}{\beta_{n-2}}\wedge \frac{d \beta_{n-1}}{\beta_{n-1}}=\alpha \frac{d \beta_1'}{\beta_1'}\wg\ldots\wg \frac{d\beta_{n-2}'}{\beta_{n-2}'} \wedge \frac{d \varphi(\tau(v_1'))}{\varphi(\tau(v_1'))}, \enspace \text{and so}$$
$$\alpha \frac{d \beta_1}{\beta_1}\wg\ldots\wg \frac{d\beta_{n-2}}{\beta_{n-2}}\wedge \frac{d \beta_{n-1}}{\beta_{n-1}} \wedge \frac{ d\beta_n}{\beta_n}=\alpha \frac{d \beta_1'}{\beta_1'}\wg\ldots\wg \frac{d\beta_{n-2}'}{\beta_{n-2}'} \wedge \frac{d \varphi(\tau(v_1'))}{\varphi(\tau(v_1'))} \wedge \frac{d \beta_n}{\beta_n}.$$	
By Lemma \ref{calcs} $(f)$, if $\varphi(\tau(v_1'))+\operatorname{N}_{F[\wp^{-1}(\alpha)]/F}(v_{0,1})=0$ then $\omega=0$, and otherwise we have
$$\omega=\alpha \frac{d \beta_1'}{\beta_1'}\wg\ldots\wg \frac{d\beta_{n-2}'}{\beta_{n-2}'} \wedge \frac{d \varphi(\tau(v_1'))}{\varphi(\tau(v_1'))} \wedge \frac{d ((\varphi(\tau(v_1'))+\operatorname{N}_{F[\wp^{-1}(\alpha)]/F}(v_{0,1})) \beta_n)}{(\varphi(\tau(v_1'))+\operatorname{N}_{F[\wp^{-1}(\alpha)]/F}(v_{0,1})) \beta_n}.$$
Consequently
$$\alpha \frac{d \beta_1}{\beta_1}\wg\ldots\wg \frac{d\beta_{n-2}}{\beta_{n-2}}\wedge \frac{d \beta_{n-1}}{\beta_{n-1}} \wedge \frac{ d\beta_n}{\beta_n}=\alpha \frac{d \beta_1}{\beta_1}\wg\ldots\wg \frac{d\beta_{n-2}}{\beta_{n-2}}\wedge \frac{d \beta_{n-1}}{\beta_{n-1}} \wedge \frac{d (\varphi(v_1))}{\varphi(v_1)}.$$
If $v_0=0$, this completes the picture.
Assume $v_0 \neq 0$.
By the assumption, there exist $\beta_1'',\dots,\beta_{n-2}''$ such that
$$\alpha \frac{d \beta_1}{\beta_1}\wg\ldots\wg \frac{d\beta_{n-2}}{\beta_{n-2}}\wedge \frac{d \beta_{n-1}}{\beta_{n-1}}=\alpha \frac{d \beta_1''}{\beta_1''}\wg\ldots\wg \frac{d\beta_{n-2}''}{\beta_{n-2}''} \wedge \frac{d \varphi(v_0)}{\varphi(v_0)}, \enspace \text{and so}$$
$$\alpha \frac{d \beta_1}{\beta_1}\wg\ldots\wg \frac{d\beta_{n-2}}{\beta_{n-2}}\wedge \frac{d \beta_{n-1}}{\beta_{n-1}} \wedge \frac{ d\varphi(v_1)}{\varphi(v_1)}=\alpha \frac{d \beta_1''}{\beta_1''}\wg\ldots\wg \frac{d\beta_{n-2}''}{\beta_{n-2}''} \wedge \frac{d \varphi(v_0)}{\varphi(v_0)} \wedge \frac{d \varphi(v_1)}{\varphi(v_1)}.$$	
By Lemma \ref{calcs} $(e)$, if $\varphi(v_0)+\varphi(v_1)=0$ then $\omega=0$, and otherwise
$$\omega=\alpha \frac{d \beta_1''}{\beta_1''}\wg\ldots\wg \frac{d\beta_{n-2}''}{\beta_{n-2}''} \wedge \frac{d (\varphi(v_0)^{-1} \varphi(v_1))}{\varphi(v_0)^{-1} \varphi(v_1)} \wedge \frac{d (\varphi(v_0)+\varphi(v_1))}{\varphi(v_0)+\varphi(v_1)}$$
and since $\varphi(v)=\varphi(v_0)+\varphi(v_1)$, this proves the statement.
\end{proof}

\begin{cor}
Using the same setting as Proposition \ref{polynomialform},
write $$V_1=\bigoplus_{0 \leq d_1,\dots,d_{n-1} \leq 1} V_{d_1,\dots,d_{n-1},1}.$$
Then for every nonzero $v_1 \in V_1$, assuming $\omega \neq 0$, we have
$$\omega=\alpha \frac{d \beta_1}{\beta_1}\wg\ldots \wedge\frac{d\beta_{n-1}}{\beta_{n-1}} \wedge \frac{d \varphi(v)}{\varphi(v)}.$$
\end{cor}

\begin{proof}
The vector space $V_1$ decomposes as $V_1' \oplus V_{0,\dots,0,1}$ where $V_1'$ is the direct sum of all $V_{d_1,\dots,d_{n-1},1}$ with $(d_1,\dots,d_{n-1}) \neq (0,\dots,0)$. Take $\tau$ to be the natural isomorphism from $V_1'$ to $V_0$. Let $v_1$ be a nonzero element in $V_1$. It can therefore be written as $v_1=v_1'+v_{0,1}$ where $v_1' \in V_1'$ and $v_{0,1} \in V_{0,\dots,0,1} (\cong F[\wp^{-1}(\alpha)])$. By the previous proposition
$$\alpha \frac{d \beta_1}{\beta_1}\wg\ldots\wg \frac{d\beta_{n-2}}{\beta_{n-2}}\wedge \frac{d \beta_{n-1}}{\beta_{n-1}}=\alpha \frac{d \beta_1'}{\beta_1'}\wg\ldots\wg \frac{d\beta_{n-2}'}{\beta_{n-2}'}\wedge \frac{d \varphi(\tau(v_1'))}{\varphi(\tau(v_1'))},$$
for some $\beta_1',\dots,\beta_{n-2}' \in F^\times$. Consequently,
$$\alpha \frac{d \beta_1}{\beta_1}\wg\ldots\wg \frac{d\beta_{n-2}}{\beta_{n-2}}\wedge \frac{d \beta_{n-1}}{\beta_{n-1}} \wedge \frac{d \beta_n}{\beta_n}=\alpha \frac{d \beta_1'}{\beta_1'}\wg\ldots\wg \frac{d\beta_{n-2}'}{\beta_{n-2}'}\wedge \frac{d \varphi(\tau(v_1'))}{\varphi(\tau(v_1'))} \wedge \frac{d \beta_n}{\beta_n}.$$
By Lemma \ref{calcs} $(f)$ we have 
$$\omega=\alpha \frac{d \beta_1'}{\beta_1'}\wg\ldots\wg \frac{d\beta_{n-2}'}{\beta_{n-2}'}\wedge \frac{d \varphi(\tau(v_1'))}{\varphi(\tau(v_1'))} \wedge \frac{d (\varphi(\tau(v_1'))+\operatorname{N}_{F[\wp^{-1}(\alpha)]/F}(v_{0,1}))\beta_n}{(\varphi(\tau(v_1'))+\operatorname{N}_{F[\wp^{-1}(\alpha)]/F}(v_{0,1}))\beta_n}$$
and so
$$\alpha \frac{d \beta_1}{\beta_1}\wg\ldots\wg \frac{d \beta_{n-1}}{\beta_{n-1}} \wedge \frac{d \beta_n}{\beta_n}=\alpha \frac{d \beta_1}{\beta_1}\wg\ldots\wg \frac{d\beta_{n-1}}{\beta_{n-1}} \wedge \frac{d \varphi(v_1)}{\varphi(v_1)}.$$
\end{proof}

\begin{cor}\label{separable}
Using the same setting as Proposition \ref{polynomialform},
let $\ell$ be an integer between $1$ and $n$. 
Write $$W=\bigoplus_{\begin{array}{r}0 \leq d_1,\dots,d_n \leq 1\\ (d_\ell,\dots,d_n) \neq (0,\dots,0)\end{array}} V_{d_1,\dots,d_n}.$$
Then for every nonzero $v \in W$, assuming $\omega \neq 0$, there exist $\beta_\ell',\dots,\beta_{n-1}' \in F$ such that
$$\omega=\alpha \frac{d \beta_1}{\beta_1}\wg\ldots \wedge \frac{d \beta_{\ell-1}}{\beta_{\ell-1}} \wedge \frac{d \beta_\ell'}{\beta_\ell'} \wedge \dots \wg \frac{d\beta_{n-1}'}{\beta_{n-1}'} \wedge \frac{d \varphi(v)}{\varphi(v)}.$$
\end{cor}

\begin{proof}
The vector space $W$ decomposes as $W_n \oplus W_{n-1} \oplus \dots \oplus W_\ell$ such that for each $k$ between $\ell$ and $n$,
$$W_k=\bigoplus_{\begin{array}{r}0 \leq d_1,\dots,d_{k-1} \leq 1\\ d_k=1, d_{k+1}=\dots=d_n=0\end{array}} V_{d_1,\dots,d_n}.$$
Let $v$ be a nonzero vector in $W$. 
Then $v$ can be written accordingly as $v=v_n+\dots+v_\ell$ where each $v_k$ belongs to $W_k$.
For each $k \in \{\ell,\dots,n\}$, by the previous lemma
$$\alpha \frac{d \beta_1}{\beta_1} \wedge \dots \wedge \frac{d \beta_{k-1}}{\beta_{k-1}} \wedge \frac{d \beta_k}{\beta_k}=
\alpha \frac{d \beta_1}{\beta_1} \wedge \dots \wedge \frac{d \beta_{k-1}}{\beta_{k-1}} \wedge \frac{d \varphi(v_k)}{\varphi(v_k)}.$$
Therefore
$$\alpha \frac{d \beta_1}{\beta_1} \wedge \dots \wedge \frac{d \beta_n}{\beta_n}=
\alpha \frac{d \beta_1}{\beta_1} \wedge \dots \wedge \frac{d \beta_{\ell-1}}{\beta_{\ell-1}} \wedge \frac{d \varphi(v_\ell)}{\varphi(v_\ell)} \wedge \dots \wedge \frac{d \varphi(v_n)}{\varphi(v_n)}.$$
Then by Proposition \ref{calcs} $(e)$ we can change the last term to be $\frac{d(\varphi(v_\ell)+\dots+\varphi(v_n))}{\varphi(v_\ell)+\dots+\varphi(v_n)}$ at the cost of possibly changing the slots $\ell$ to $n-1$.
\end{proof}

\section{Linkage of Decomposable Differential Forms}

\begin{thm}\label{collections}
Given a field $F$ of $\operatorname{char}(F)=p$, a positive integer $n$ and an integer $\ell \in \{1,\dots,n\}$, every collection of $1+\sum_{i=\ell}^n 2^{i-1}$ inseparably $n$-linked decomposable differential forms in $H_p^{n+1}(F)$ are separably $\ell$-linked as well.
\end{thm}

\begin{proof}
Write $m=2^{n-1}+\dots+2^{\ell-1}$.
Let $\left\{\alpha_i \frac{d \beta_1}{\beta_1} \wedge \dots \wedge \frac{d \beta_n}{\beta_n} : i \in \{0,\dots,m\}\right\}$ be the collection of $m+1$ inseparably $n$-linked decomposable differential forms in $H_p^{n+1}(F)$ under discussion.
The number of $n$-tuples $(d_1,\dots,d_n)$ with $(d_\ell,\dots,d_n) \neq (0,\dots,0)$ is $m$.
We denote arbitrarily the elements of the set $\{\beta_1^{d_1} \cdot \ldots \cdot \beta_n^{d_n} : 0 \leq d_1,\dots,d_n \leq 1, (d_\ell,\dots,d_n) \neq (0,\dots,0)\}$ by $\gamma_1,\dots,\gamma_m$ (with possible repetitions). 
Recall that the norm of an element $x+\lambda y$ in $F[\lambda : \lambda^p-\lambda=\alpha]$ is $x^p-x y^{p-1}+y^p \alpha$.
Therefore by Corollary \ref{separable}, for every $i \in \{0,\dots,m\}$ and any choice of $x_{i,1},y_{i,1},\dots,x_{i,m},y_{i,m} \in F$ with $(x_{i,1},\dots,y_{i,m}) \neq (0,\dots,0)$ we can change the form $\alpha_i \frac{d \beta_1}{\beta_1} \wedge \dots \wedge \frac{d \beta_n}{\beta_n}$ to $\alpha_i \frac{d \beta_1}{\beta_1} \wedge \dots \wedge \frac{d \beta_{\ell-1}}{\beta_{\ell-1}} \wedge \frac{d b_{i,ell}}{b_{i,ell}} \wedge \dots \wedge \frac{d b_{i,n-1}}{b_{i,n-1}} \wedge \frac{d \delta_i}{\delta_i}$ where 
$$\delta_i=\gamma_1 (x_{i,1}^p-x_{i,1} y_{i,1}^{p-1}+\alpha_i y_{i,1}^p)+\dots+\gamma_m (x_{i,m}^p-x_{i,m} y_{i,m}^{p-1}+\alpha_i y_{m,1}^p)$$
for some $b_{i,\ell},\dots,b_{i,n-1} \in F$.

Therefore, in order to show that the forms in the collection are separably $\ell$-linked, it is enough to show that the following system of $m$ equations in $2m(m+1)$ variables has a solution:
\begin{eqnarray*}
\alpha_0+\delta_0 & = &\alpha_1+\delta_1\\
\alpha_0+\delta_0 & = &\alpha_2+\delta_2\\
& \vdots & \\
\alpha_0+\delta_0 & = &\alpha_m+\delta_m\\
\end{eqnarray*}
If we take $x_{0,i}=x_{1,i}=\dots=x_{m,i}$, $y_{0,i}=1$ and $y_{i,i}=0$ and $y_{i,j}=1$ for all $i,j \in \{1,\dots,m\}$ with $i \neq j$, then the $i$th equation in this system becomes a linear equation in one variable $x_{0,i}$ (whose coefficient is $\gamma_i$).
This system therefore has a solution.
\end{proof}

\begin{cor}
Given a field $F$ of $\operatorname{char}(F)=2$, a positive integer $n$ and an integer $\ell \in \{1,\dots,n\}$, every collection of  $1+\sum_{i=\ell}^n 2^{i-1}$ inseparably $n$-linked quadratic $(n+1)$-fold Pfister forms are also separably $\ell$-linked. 
\end{cor}

\begin{proof}
Follows Theorem \ref{collections} and the identification of quadratic Pfister forms with decomposable differential forms appearing in \cite{Kato:1982}.
\end{proof}

By the identification of octonion algebras with their 3-fold Pfister norm forms (\cite[Theorem 33.19]{BOI}), we conclude that if three octonion $F$-algebras share a biquadratic purely inseparable field extension of $F$ then they also share a quaternion subalgebra.
(Recall that when $\operatorname{char}(F)=2$, an octonion algebra $A$ over $F$ is of the form $Q+Qz$ where $Q=[\alpha,\beta)_{2,F}$ is a quaternion algebra, $z^2=\gamma$ and $z \ell=\ell^\sigma z$ for every $\ell \in Q$ where $\sigma$ is the canonical involution on $Q$, $\alpha \in F$ and $\beta,\gamma \in F^\times$. In particular, when $A$ is a division algebra, $F[\sqrt{\beta},\sqrt{\gamma}]$ is a subfield of $A$ and a biquadratic purely inseparable field extension of $F$.)

It is not known to the authors if the sizes of the collections mentioned in Theorem \ref{collections} are sharp in the sense that larger collections of inseparably $n$-linked differential forms in $H_p^{n+1}(F)$ need not be separably $\ell$-linked.
Even the very special case of three inseparably linked quaternion algebras would be very interesting, in either direction.

\begin{ques}
Are every three inseparably linked quaternion algebras also separably linked?
\end{ques}

\section{Vanishing Cohomology Groups}

One can use Proposition \ref{polynomialform} to show that for fields $F$ with finite $u_p(F)$ the cohomology groups $H_p^{n+1}(F)$ vanish from a certain point and on. See section \ref{up_invariant} for the definition of $u_p(F)$.

\begin{thm}\label{Cohomological}
Given a field $F$ with $\operatorname{char}(F)=p$, for any $n$ with $(2^n-1)p\geq u_p(F)$, $H^{n+1}_p(F)=0$.
\end{thm}

\begin{proof}
Let $\omega=\alpha \frac{d \beta_1}{\beta_1}\wg\ldots\wg \frac{d\beta_{n-1}}{\beta_{n-1}}\wedge \frac{d \beta_n}{\beta_n}$ be an arbitrary decomposable form in $H_p^{n+1}(F)$.
Consider the direct sum $\Phi$ of the form $\varphi$ from Proposition \ref{polynomialform} and the two-dimensional form 
$\psi(x,y)=\alpha x^p-x^{p-1} y+y^p$.
Note that by Lemmas \ref{closure}, \ref{directsum} and \ref{twodim}, $\Phi$ is $p$-regular.
If $(2^n-1) p+1 \geq u_p(F)$ then $\dim(\Phi)=(2^n-1) p+2$, hence $\Phi$ is isotropic, which means that there exist $(x_0,y_0,v_0) \neq 0$ such that $\psi(x_0,y_0)+\varphi(v_0)=0$.
If $v_0=0$ then it means $\alpha \in \wp(F)$, and therefore $\omega$ is trivial.
If $\varphi$ is isotropic then $\omega$ is trivial by Proposition \ref{polynomialform}.
Otherwise, $\varphi(v_0) \neq 0$.
By proposition \ref{polynomialform}, $\omega=\alpha \frac{d \beta_1'}{\beta_1'} \wedge \dots \wedge \frac{d \beta_{n-1}'}{\beta_{n-1}'} \wedge \frac{d \varphi(v_0)}{\varphi(v_0)}$ for some $\beta_1',\dots,\beta_{n-1}' \in F$.
If $x_0=0$ then $\Phi(0,y_0,v_0)=y_0^p+\varphi(v_0)=0$, and so $\varphi(v_0)=(-y_0)^p$, which means $\omega$ is trivial.
If $x_0 \neq 0$ then $\Phi(1,\frac{y_0}{x_0},\frac{v_0}{x_0})=0$ as well.
By Proposition \ref{polynomialform}, $\omega=\alpha  \frac{d \beta_1'}{\beta_1'} \wedge \dots \wedge \frac{d \beta_{n-1}'}{\beta_{n-1}'} \wedge \frac{d \varphi(\frac{v_0}{x_0})}{\varphi(\frac{v_0}{x_0})}$ for some $\beta_1',\dots,\beta_{n-1}' \in F$.
Then, by Lemma \ref{calcs} $(a)$, we can add $\varphi(\frac{v_0}{x_0})$ to $\alpha$, but then the coefficient is congruent to $\Phi(1,\frac{y_0}{x_0},\frac{v_0}{x_0})$ modulo $\wp(F)$, and therefore $\omega$ is trivial.
\end{proof}

\begin{cor}
If $F$ is a $\widetilde{C}_{p,m}$ field with $\operatorname{char}(F)=p$ then for all $n \geq \lceil (m-1) \log_2(p) \rceil+1$, $H_p^{n+1}(F)=0$.
\end{cor}

\begin{proof}
Since $F$ is a $\widetilde{C}_{p,m}$ field, $u_p(F) \leq p^m$.
If $n \geq \lceil (m-1) \log_2(p) \rceil+1$ then 
$2^n \geq p^{m-1}+1$, and so $2^n \cdot p \geq p^m+p$, which means that
$2^n \cdot p-p \geq p^m$, and therefore $(2^n-1) p+1 \geq p^m+1 \geq u_p(F)+1$.
\end{proof}

\begin{exmpl}
If we take $m=1$ then $\lceil (m-1) \log_2(p) \rceil+1=1$, and therefore $H_p^{1+1}(F)=H_p^2(F)=0$ for $\widetilde C_{p,1}$ fields $F$.
The group $H_p^2(F)$ is isomorphic to $\prescript{}{p}Br(F)$, which is generated by symbol algebras $[\alpha,\beta)_{p,F}$.
Another way of seeing that $\prescript{}{p}Br(F)$ is trivial for $\widetilde C_{p,1}$ is to consider an arbitrary generator $A=[\alpha,\beta)_{p,F}=F \langle x,y :x^p-x=\alpha, y^p=\beta, y x y^{-1}=x+1 \rangle$ and to show that it must be split.

If $F[x]$ is not a field then this algebra is split, so suppose this is a field.
Then consider the norm form $N : F[x] \rightarrow F$.
For any choice of $f \in F[x]^\times$, the element $z=f y$ satisfies $z^p=N(f) \beta$ and $z x z^{-1}=x+1$, so $A=[\alpha,N(F) \beta)_{p,F}$.
The element $w=x+z$ satisfies $w^p-w=\alpha+N(f) \beta$ (see \cite[Lemma 3.1]{Chapman:2015}).
Now, consider the homogeneous polynomial form 
\begin{eqnarray*}
\varphi &:& F \times F \times F[x] \rightarrow F\\
& & (a,b,f) \mapsto \alpha a^p-a^{p-1} b+b^p+N(f) \beta.
\end{eqnarray*}
This form is $p$-regular by Lemma \ref{fieldextension}, Remark \ref{scalar} and Lemma \ref{twodim}.
It is isotropic because its dimension is $p+2$ and $F$ is $\widetilde  C_{p,1}$,
i.e. there exist $a_0,b_0 \in F$ and $f_0 \in F[x]$ not all zero such that $\varphi(a_0,b_0,f_0)=0$.
The case of $a_0=b_0=0$ is impossible because then $N(f_0)=0$ which means that $F[x]$ is not a field, contrary to the assumption.
If $a_0=0$ then the element $t=b_0+f_0 y$ satisfies $t^p=0$, which means $A$ is split.
If $a_0 \neq 0$ then the element $t=x+\frac{b_0}{a_0}+\frac{1}{a_0} f_0 y$ satisfies $t^p-t=0$, which again means $A$ is split.
\end{exmpl}

Recall that for $p=2$, the notion of $u_p(F)$ coincides with the $u$-invariant of $F$.
Then the statement recovers the known fact that if $n$ satisfies $2^{n+1} > u(F)$ then $H_2^{n+1}(F)=0$.

\section*{Bibliography}
\bibliographystyle{amsalpha}
\bibliography{bibfile}

\end{document}